\documentclass[11pt]{amsart}
\usepackage{amsmath, amsthm, amssymb}
\usepackage{subcaption}
\usepackage{graphicx}
\usepackage{mathbbol}
\usepackage{txfonts}
\usepackage[usenames]{color}
\usepackage{srcltx} 
\usepackage{xpatch}

\xpatchcmd{\proof}
  {\itshape}
  {\bfseries}
  {}
  {}

\newtheorem{thm}{Theorem}[section]

\newtheorem{cor}[thm]{Corollary}   
\newtheorem{lem}[thm]{Lemma}   
\newtheorem{prop}[thm]{Proposition}
\newtheorem{defn}[thm]{Definition}

\newcommand{\be}{\begin{equation}}
\newcommand{\ee}{\end{equation}}

\newcommand{\ptGHto}{\stackrel { \textrm{ptGH}}{\longrightarrow} }
\newcommand{\GHto}{\stackrel { \textrm{GH}}{\longrightarrow} }

\newcommand{\dil}{\textrm{dil}}

\begin{document}

\title{The Checkered Smocked Space and its Tangent Cone}

\author{Victoria Antonetti\\Maziar Farahzad \\Ajmain Yamin}

\thanks{}




\begin{abstract} 
Smocked metric spaces were first defined in \cite{Sormani-Kazaras-Students1} and it was proved that if a norm on the Euclidean space uniformly estimates the pseudometric of a smocked metric space, then the tangent cone at infinity is unique and is a norm vector space with that estimating norm. In this paper, we explicitly calculate the norm approximating the pseudometric of the checkered smocked space and find the tangent cone at infinity.
\end{abstract}

\maketitle
\section{Introduction}

The notion of a smocked metric space was first introduced in \cite{Sormani-Kazaras-Students1}. In that paper, a variety of smocked metric spaces were considered and proved to have unique tangent cones at infinity.
As the main theorem, it was proved that if a norm on the Euclidean space uniformly estimates the pseudometric of a smocked metric space, then the tangent cone at infinity is unique and is a norm vector space with that estimating norm. The formal theorem is in section \ref{sec:tancone}. 

In this paper, we consider the checkered smocked space $X_H$, and explicitly find its estimating norm and prove it too has a unique tangent cone at infinity. The rest of the paper is organized as follows. In section \ref{sec:Background}, we review the definition of a smocked metric space, the definition of checkered smocked space $X_H$, and the definition of a tangent cone at infinity of a smocked metric space. In section \ref{sec:length}, we estimate the length of geodesics in $X_H$ to find the estimating norm. Finally, the uniqueness of the tangent cone at infinity of $X_H$ is proved in section \ref{sec:tancone}.  This tangent cone at infinity is a normed plane with norm $F:\mathbb{E}^2\rightarrow[0,\infty)$ given by 
\be\label{equ:norm}
F(x_1, x_2) = \frac{\sqrt 2 }{3}(|x_1|+|x_2|) + \frac{2-\sqrt 2}{3}||x_1|-|x_2||.
\ee

\section{\label{sec:Background}Background}

We recall the definition of a smocked metric space and the checkered smocked space $X_H$, as introduced in \cite{Sormani-Kazaras-Students1}.
\begin{defn} \label{defn-smock}
Given a Euclidean space, $\mathbb{E}^N$, and a finite or countable collection of
disjoint connected compact sets called {\bf smocking intervals} or {\bf smocking stitches}, 
\be
\mathcal{I}=\{I_j: \, j \in J\},
\ee
with a positive {\bf smocking separation factor},
\be \label{s-factor}
\delta=\min\{|z-z'|: \, z\in I_j, \, z'\in I_{j'},\, j\neq j' \in J\} >0,
\ee
one can define the {\bf smocked
metric space}, $(X,d)$, in which each stitch is viewed as a single point.
\be
X = \left\{ x: \, x \in {\mathbb{E}^N}\setminus S\right\} \cup \mathcal{I}
\ee
where $S$ is the {\bf smocking set} or {\bf smocking pattern}:
\be
S= \bigcup_{j \in J} I_j .
\ee
We have a {\bf smocking map} 
$\pi: \mathbb{E}^N \to X $ defined by
\be
 \pi(x) = \begin{cases} 
          x & \textrm{ for }x \in {\mathbb{E}^N}\setminus S \\
          I_j&  \textrm{ for } x\in I_j \textrm{ and } j\in J.
                 \end{cases}
\ee
The {\bf smocked distance function}, $d: X\times X \to [0,\infty)$, is defined for $y, x\notin \pi(S)$, and stitches
$I_m$ and $I_k$ as follows:
\begin{eqnarray*}
d(\,x,\,y\,) &=& \min \left\{d_0(x,y), d_1(x,y), d_2(x,y), d_3(x,y), ...\right\} \\
d(\,x,\, I_k) &=& \min \{ d_0(x,z),  d_1(x,z), d_2(x,z), d_3(x,z), ...:\, z\in I_k\} \\
d(I_m,I_k) &=& \min \{ d_0(z',z), d_1(z',z), d_2(z',z), d_3(z',z), ... \,:z'\in I_m,\, z\in I_k \}
\end{eqnarray*}
where $d_k$ are the sums of lengths of segments that jump to and between $k$ stitches:
\begin{eqnarray*}
d_0(v,w) &=& |v-w|\\
d_1(v,w) &=&  \min\{ |v-z_1|+|z'_1-w|:\, z_1, z_1'\in I_{j_1}, \, j_1 \in J\}\\
d_2(v,w) &=& \min\{ |v-z_1|+|z'_1-z_2|+|z'_2-w|:\, z_i, z'_i\in I_{j_i}, \, j_1\neq j_2 \in J\}\\
d_k(v,w) &=& \min \{  |v-z_1|+\sum_{i=1}^{k-1} |z'_i-z_{i+1}|+|z'_k-w|:\, z_i, z'_i\in I_{j_i}, \, j_1\neq \cdots \neq j_k \in J.
\end{eqnarray*}
We define the {\bf smocking pseudometric} $\bar{d}: {\mathbb{E}^N}\times {\mathbb{E}^N} \to [0, \infty)$
to be
$$
\bar{d}(v,w)= d(\pi(v), \pi(w))=\min \{d_k(v',w'): \,\pi(v)=\pi(v'),\, \pi(w)=\pi(w'),\, k\in {\mathbb{N}}\}.
$$
We will say the smocked space is {\bf parametrized by points in the stitches}, if 
\be\label{param-by-points}
J \subset {\mathbb{E}}^N \textrm{ and } \forall j \in J \,\, j \in I_j.
\ee
\end{defn}

\begin{defn}\label{defn-H}
The {\bf checkered smocked space} $(X_H, d_H)$ is a smocked plane defined as in Definition 3.1 of \cite{Sormani-Kazaras-Students1}. 
We start with the Euclidean plane ${\mathbb{E}}^2$.
We define our index set:
$$
J_H = J_H^- \cup J_H^\vert
$$
where
$$
J_H^-= 3\mathbb{Z} \times 3\mathbb{Z}
$$
and
$$
J_H^\vert=(3\mathbb{Z} +1.5) \times (3\mathbb{Z} + 1.5)
$$

We define our intervals which are horizontal (of length 1):
$$
I_{(j_1,j_2)}= [j_1-0.5,j_1+0.5]\times \{j_2\} \textrm{ when } (j_1,j_2) \in J_H^-
$$
and vertical (of length 1):
$$
I_{(j_1,j_2)}= \{j_1\}\times [j_2-0.5,j_2+0.5] \textrm{ when } (j_1,j_2) \in J_H^\vert
$$

See Figure \ref{fig:patternH}.

\begin{figure}[h]
\includegraphics[scale=.1]{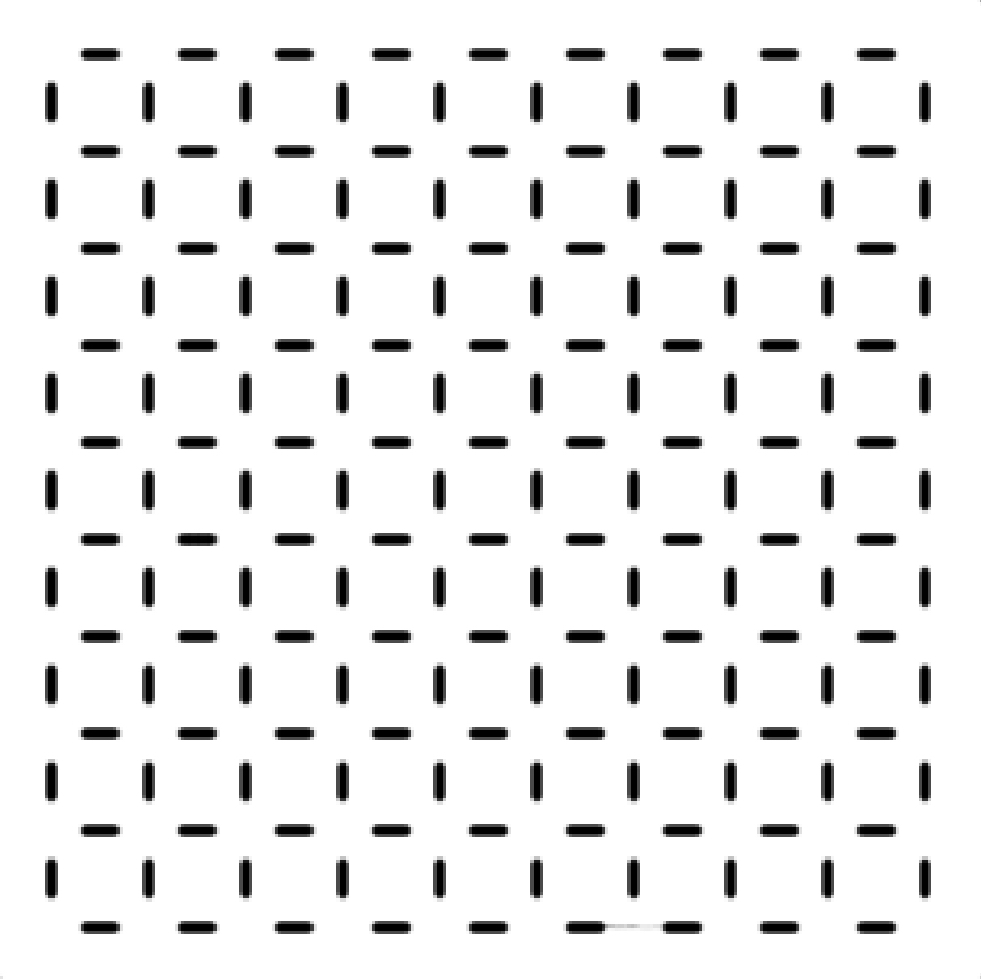}
\caption{Smocking Pattern H.}
\label{fig:patternH}
\end{figure}
\end{defn}

Gromov-Hausdorff convergence was first defined by David Edwards in \cite{Edwards}. It was
rediscovered by Gromov in \cite{Gromov-metric}. The text \cite{BBI} gives an excellent introduction to this topic.  

\begin{defn} \label{defn-GH}
We say a sequence of compact metric spaces
\be
(X_j, d_j) \GHto (X_\infty, d_\infty)
\ee
if and only if
\be
d_{GH}((X_j,d_j), (X_\infty, d_\infty)) \to 0,
\ee
where the Gromov-Hausdorff distance is defined by
\be
d_{GH}(X_j, X_\infty) = \inf \{d^Z_H(\varphi_j(X_j), \varphi_\infty(X_\infty)): \,\, Z,\,\, \varphi_j: X_j \to Z\}
\ee
where the infimum is over all compact metric spaces, $Z$, and over all
distance preserving maps $\varphi_j: X_j\to Z$:
\be
d_Z(\varphi_j(a), \varphi_j(b))=d_j(a,b) \,\,\,\forall a,b \in X_j.
\ee
The Hausdorff distance is defined by
\be
d_H(A_1, A_2) = \inf\{r:\,\, A_1\subset T_r(A_2) \textrm{ and } A_2\subset T_r(A_1) \}.
\ee
\end{defn}

Since we are considering unbounded metric spaces, we must consider the following definition by Gromov:

\begin{defn}\label{defn-ptGH}
If one has a sequence of complete noncompact metric spaces, $(X_j, d_j)$, and points
$x_j \in X_j$, one can define pointed GH convergence:
\be
(X_j, d_j,x_j) \ptGHto (X_\infty, d_\infty, x_\infty)
\ee
if and only if
for every radius $r>0$, the closed balls of radius $r$ in $X_j$ converge in the GH sense as metric spaces
with the restricted distance to closed balls in $X_\infty$:
\be
d_{GH}((\bar{B}_r(x_j)\subset X_j,d_j), (\bar{B}_r(x_\infty)\subset X_\infty, d_\infty)) \to 0.
\ee
\end{defn}

The idea behind the tangent cone at infinity of a metric space is probing the asymptotic behavior of the metric space as one zooms out. For this purpose, we consider a sequence of rescalings of a metric space and see if this sequence or a subsequence of it converges in GH sense. If it does, then one obtains a space which is a tangent cone at infinity.
\begin{defn} \label{defn-tan-cone}
A complete noncompact metric space with infinite diameter, $(X, d_X)$,
has a tangent cone at infinity, $(Y, d_Y)$, if there is a sequence of
rescalings, $R_j \to \infty$, and points, $x_0\in X$ and $y_0\in Y$, such that
\be
(X, d/R_j, x_0) \ptGHto (Y, d_Y, y_0).
\ee
\end{defn}

\section{\label{sec:length}Length of Geodesics in $X_H$}

\subsection{Preliminary definitions and notation}

\begin{defn}\label{defn:path}
A {\bf path} from $a$ to $b$ in $X_H$ is a list of directed segments $P=(P_i)_{i=1}^n$ in $\mathbb{E}^2$ where
\begin{itemize}
    \item the initial point of $P_1$ is $\pi^{-1}(a)$
    \item the terminal point of $P_n$ is $\pi^{-1}(b)$
    \item the terminal point of $P_i$ lies in the same smocking stitch as the initial point of $P_{i+1}$ for all $i = 1,\dots,n-1$.
\end{itemize}

The number $n$ is called the {\bf combinatorial length} of $P$. The {\bf length} of $P$, denoted $L(P)$, is the sum of the euclidean lengths of the segments $P_i$.  A {\bf geodesic} in $X_H$ is a path from $a$ to $b$ such that $L(P) = d_H(a,b)$.
\end{defn}

\begin{defn}

For each $j\in J$
        \begin{itemize}
            \item let $\nearrow_j$ be the shortest directed segment from $I_{j+(-1.5,-1.5)}$ to $I_j$
            \item let $\searrow_j$ be the shortest directed segment from $I_{j+(-1.5,1.5)}$ to $I_j$
            \item let $\nwarrow_j$ be the shortest directed segment from $I_{j+(1.5,-1.5)}$ to $I_j$
            \item let $\swarrow_j$ be the shortest directed segment from $I_{j+(1.5,1.5)}$ to $I_j$
        \end{itemize}
For each $j\in J_{-}$
        \begin{itemize}
            \item let $\rightarrow_j$ be the shortest directed segment from $I_{j+(-3,0)}$ to $I_j$
            \item let $\leftarrow_j$ be the shortest directed segment from $I_{j+(3,0)}$ to $I_j$
        \end{itemize}
For each $j\in J_{\vert}$
        \begin{itemize}
            \item let $\uparrow_j$ be the shortest directed segment from $I_{j+(0,-3)}$ to $I_j$
            \item let $\downarrow_j$ be the shortest directed segment from $I_{j+(0,3)}$ to $I_j$
        \end{itemize}

We call $\nearrow_j$, $\nwarrow_j$, $\searrow_j$, and $\swarrow_j$ {\bf diagonals}.
We call $\leftarrow_j$ and $\rightarrow_j$  {\bf horizontals}.
We call $\uparrow_j$ and $\downarrow_j$  {\bf verticals}.
We refer to the collection of all such {\bf network parts} as the {\bf network} $N$, $$N = \bigcup_{j\in J} \{\nearrow_j,\nwarrow_j, \searrow_j,\swarrow_j \} \cup \bigcup_{j\in J_{-}} \{\leftarrow_j, \rightarrow_j \} \cup \bigcup_{j\in J_{\vert}} \{\uparrow_j, \downarrow_j\}$$ 
\end{defn}

\begin{defn}
A {\bf network path} is a path which consists of network parts.
\end{defn}

\begin{prop}\label{prop-networkPathLength}
If $P$ is a network path with $D$ number of diagonals and $P$ has combinatorial length $n$ then the length of $P$ is $$L(P) = D\sqrt 2 + 2(n-D).$$
\end{prop}

When decomposing a network path with known initial point into network parts, we do not need to label the index of each particular network part. The terminal end of one network part determines the beginning end of the next.  For example the network path from $I_0$ to $I_{(6,3)}$ given by $\nearrow_{(1.5,1.5)} \nearrow_{(3,3)} \rightarrow_{(6,3)}$ may be denoted by $\nearrow \nearrow \rightarrow$, provided we specify that the starting point is $I_0$.  Further more, write $\nearrow^2$ to denote $\nearrow \nearrow$, etc, and define $\nearrow^0$ to be the stationary path, etc.
\begin{defn}
The {\bf Awesome path} $A_j$ from $I_0$ to $I_j$ is constructed as follows:
\begin{itemize}
\item if $0=j_1 < j_2$ then $A_j= \nearrow \uparrow^{\frac{j_2}{3}-1}\nwarrow$ 
\item if $1.5 \leq j_1\leq j_2$, $A_j = \nearrow \uparrow^{\frac{j_2 - j_1}{3}}\nearrow^{\frac{2}{3}j_1-1}$
\item if $0\leq j_2 < j_1 $ then $A_j = \rightarrow^{\frac{j_1 - j_2}{3}}\nearrow^{\frac{2}{3}j_2}$
\item if $j_2<0\leq j_1$ then $A_j = R_X\left(A_{(|j_1|,|j_2|)}\right)$ where $R_X$ is reflection over the X-axis
\item if $j_1<0$ then $A_j = R_Y\left(A_{(|j_1|,j_2)}\right)$ where $R_Y$ is reflection over the Y-axis
\end{itemize}
as in figure \ref{Awesome_paths}.
\end{defn}
 \begin{figure}[h!]
            \centering
            \includegraphics[scale=.2]{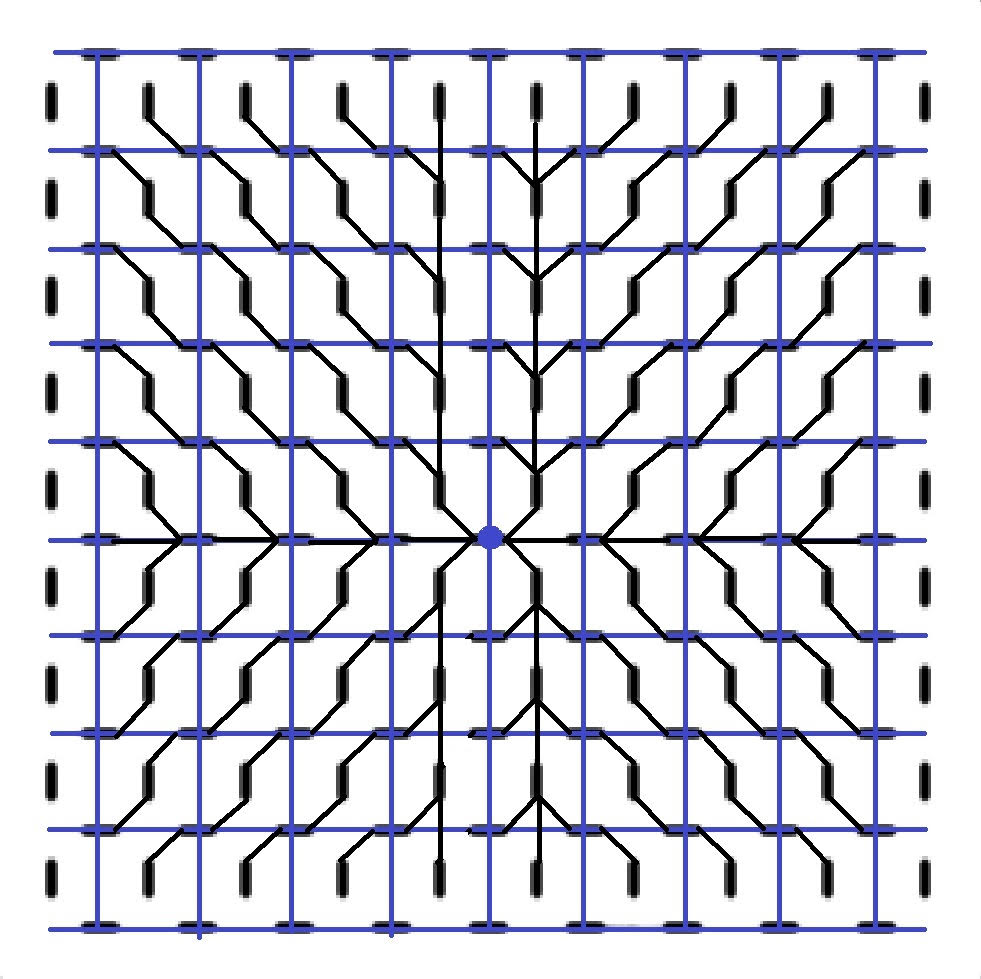}
            \caption{Awesome Paths}
            \label{Awesome_paths}
        \end{figure} 
\begin{prop}
The length of an Awesome path $A_j$ is $$L(A_j) = 
\begin{cases}
      \frac{2\sqrt{2}}{3} \min (|j_1|,|j_2|)+ \frac{2}{3}||j_1|-|j_2|| & \text{if } j_1 \neq 0,\\
      2\sqrt{2} + \frac{2}{3}|j_2| -2 & \text{if } j_1 = 0.\\
\end{cases}$$
\end{prop}

\begin{defn}\label{def-generalized-Awesome}
The {\bf Awesome path} $A_k^j$ from $I_j$ to $I_k$ is constructed as follows:\\
 Let $T_1$ be an isometry of the Euclidean plane given by 
$$T_1= \begin{cases}
      R_D & \text{if $j\in J_J^\vert$,}\\
      id & \text{otherwise.}
\end{cases}$$

Now let $T_2$ denote the translation map on the Euclidean plane taking $T_1(j)$ to $0$. Let $A$ be the Awesome path from $I_0$ to $I_{(x,y)}$ where $(x,y) = T_2T_1(j)$.  We define the $i$'th part of $A_k^j$ as $$(A_k^j)_i = T_1^{-1}T_2^{-1}(A_i).$$

\end{defn}

\begin{prop}
The length of an Awesome path $A_k^j$ is $$L(A_k^j) = 
\begin{cases}
      \frac{2\sqrt{2}}{3} \min (|j_1-k_1|,|j_2-k_2|)+ \frac{2}{3}||j_1-k_1|-|j_2-k_2|| & \text{if } j_1 \neq 0,\\
      2\sqrt{2} + \frac{2}{3}|j_2-k_2| -2 & \text{if } j_1 = k_1.\\
\end{cases}$$
\end{prop}

\subsection{Awesome Paths are more efficient than Euclidean Paths}

\begin{defn}
The {\bf Euclidean Path} from interval $I_j$ to $I_k$ in $X_H$, denoted $E_k^j$, is the path consisting of a single directed segment which minimizes the Euclidean length.  Notice the length of $E_k^j$, as defined in Definition \ref{defn:path} is $L(E_k^j) = d_E(I_j,I_k)$ where $d_E$ is the Euclidean metric on $\mathbb{E}^2$. 
\end{defn}

\begin{thm} \label{thm-faster}
The length of the Awesome path from interval $I_j$ to $I_k$ is less than or equal to the length of the Euclidean of the Euclidean path from interval $I_j$ to $I_k$:
$$L(A_k^j) \leq L(E_k^j).$$
\end{thm}

The proof of Theorem \ref{thm-faster} will be done in 3 steps.  First we prove Lemma \ref{lem-H-ineq-1} which handles the case $j\in J_{|}$. Then we prove Lemma \ref{lem-H-ineq-2} which handles the case $j\in J_{-}$, $j_1 \neq 0$. Finally we prove Lemma \ref{lem-H-ineq-3} which handles the case $j_1=0$. 

\begin{lem}\label{lem-H-ineq-1}
For any $x,y\in \mathbb{R}$, $1.5 \leq|x| $, and $1.5\leq |y|$,
$$\left(\frac{2\sqrt{2}}{3} \min (|x|,|y|)+ \frac{2}{3}||x|-|y||\right)^2 \leq \left(|x|-\frac{1}{2}\right)^2 + \left(|y|-\frac{1}{2}\right)^2.$$
This become equality only when $1.5= |x| = |y|$.
\end{lem}
\begin{proof}
It is enough to show \[
\left(\frac{2\sqrt{2}}{3} \left(x+ \frac{1}{2}\right)+ \frac{2}{3}(y-x)\right)^2 \leq x^2 + y^2 \tag{A}
\] is true for all $1\leq x\leq y$. 
This is because for any $x,y\in \mathbb{R}$ with $1.5 \leq|x| $, and $1.5\leq |y|$ we observe that $(x',y') = \left(\min\left(|x|,|y|\right)-\frac{1}{2}, \max\left(|x|,|y|\right)-\frac{1}{2}\right)$ satisfies $1\leq x'\leq y'$, and plugging $(x',y')$ into inequality A gives the desired inequality.

Let $r\in \mathbb{R}$.
$S_r = \{(x,y)\in \mathbb{E}^2 \vert 1 \leq x \leq y, LHS(x,y) = r^2\}$.  $S_r$ is a line segment with closed endpoints $P = \left(\frac{3r-\sqrt 2}{2\sqrt 2},\frac{3r-\sqrt 2 }{2 \sqrt 2}\right)$ and $Q=\left(1,\frac{3r}{2} - \frac{3\sqrt{2}}{2} + 1\right)$.  Now consider $C_r = \{(x,y)\in \mathbb{E}^2 \vert 1 \leq x \leq y, RHS(x,y) = r^2\}$.  $C_r$ is an arc of a circle. We claim for any point $(x,y)$ of the segment $S_r$, the connecting segment between $(x,y)$ and the origin intersects the arc $C_r$ at a point. Equivalently, the entire line segment $S_r=[P,Q]$ lies outside $B_r(0)$, the Euclidean open ball of radius $r$ centered at the origin; i.e $[P,Q] \cap B_r(0) = \varnothing$. This implies, $LHS(x,y) \leq RHS(x,y)$.

We consider the half-plane $H=\{(x,y) + (0, \lambda) \vert (x,y) \in L, \lambda \geq 0\} = \{(x,y) + (\lambda,  \lambda) \vert (x,y) \in L, \lambda \geq 0\}$ where $L$ is the line $L=\{(x,y)\in \mathbb{E}^2 \vert y = -x + r\sqrt 2\}$. Since $P=\left(\frac{3r-\sqrt 2}{2\sqrt 2}, \frac{3r-\sqrt 2}{2\sqrt 2}\right) = (\frac{r\sqrt 2}{2},\frac{r\sqrt 2}{2}) + (\lambda,\lambda)$ where $\lambda = \frac{1}{2}\left(\left(\frac{3}{\sqrt 2}-\sqrt 2\right)r -1\right) \geq 0$, we have $P\in H$.  Since $Q = (1,\frac{3r}{2}-\frac{3\sqrt 2}{2} + 1) = (1, 
-1+r\sqrt 2) + (0,\lambda)$ where $\lambda = \frac{3r}{2}-\frac{3\sqrt 2}{2} + 1 +1 - r\sqrt 2 \geq 0$, we have $Q\in H$.  Since $H$ is convex, $[P,Q] \subset H$.  Therefore $[P,Q] \cap H^c = \varnothing$.

Note $\mathbb{E}^2 = H \sqcup H^c$. It is easy to show that $B_r(0) \subset H^c$.  Finally, $[P,Q] \cap B_r(0) = \varnothing$.
\end{proof}

\begin{lem}\label{lem-H-ineq-2}
For any $(x,y)\in J_{-}$, $3 \leq |x|$
$$\left(\frac{2\sqrt{2}}{3} \min (|x|,|y|)+ \frac{2}{3}||x|-|y||\right)^2 \leq \left(|x|-1\right)^2 +|y|^2.$$
This becomes equality only when $(x,y) = (3,0)$ or $(-3,0)$.
\end{lem}

\begin{proof}
 If $y = 0$, the inequality simplifies to $\frac{2}{3}|x| \leq |x|-1$ which is true for all $3 \leq |x|$.
 
 Suppose $y\neq 0$, so $3\leq |y|$.  In order to prove the , it is enough to show \[\left(\frac{2\sqrt{2}}{3}x+ \frac{2}{3}(y-x)\right)^2 \leq \left(x-1\right)^2 +y^2\tag{B1}
\] is true for all $3\leq x \leq y$ and 
\[\left(\frac{2\sqrt{2}}{3}y+ \frac{2}{3}(x-y)\right)^2 \leq \left(x-1\right)^2 +y^2\tag{B2}
\] is true for all $3\leq y \leq x$.

We start with inequality B1.  $$RHS - LHS = Ax^2 + By^2 - Cxy - Dx + E $$ where $A = 1 - \frac{8}{9} - \frac{4}{9} + \frac{8\sqrt 2}{9}$, $B = \frac{5}{9}$, $C = \frac{8}{9}(\sqrt 2 -1)$, $D = 2$, $E=1$.  Notice $A, B, C, D, E > 0$.  Since $0\leq x \leq y$ $$Ax^2 + Bx^2 -Cy^2 - Dy + E \leq RHS - LHS$$
Let $f(x) = (A+B)x^2 + E$ and $g(y) = Cy^2 + Dy$.  Notice $f(x) > g(x)$, $\forall x, y\in \mathbb{R}$.  Thus $$0 < f(x) - g(y) \leq  RHS - LHS.$$  Hence $RHS > LHS$, as required.

Now consider inequality B2.  
$$RHS - LHS = Ax^2 + By^2 - Cxy -Dx + E$$ where $A=\frac{5}{9}$, $B=\frac{5}{9} + \frac{8}{9}(\sqrt 2 -1)$, $C = \frac{8}{9}(\sqrt 2 - 1)$, $D=2$, $E=1$.  Notice $A,B,C,D,E > 0$.  Since $0\leq y\leq x$ 
$$Ay^2 + By^2 - Cx^2 -Dx + E \leq RHS - LHS$$  Let $f(y) = (A+B)y^2 + E$ and $g(x) = Cx^2 +Dx$.  Notice $f(y) > g(x)$, $\forall x, y\in \mathbb{R}$.  Thus $$0 \leq f(y) - g(x) \leq RHS - LHS$$ Hence $RHS > LHS$, as required.
\end{proof}

\begin{lem}\label{lem-H-ineq-3}
For any $y\in \mathbb{R}$, $3\leq|y|$
$$2\sqrt{2} + \frac{2}{3}|y| -2 < |y|.$$
\end{lem}

\begin{proof}[Proof of Theorem \ref{thm-faster}]
By Lemma \ref{lem-H-ineq-1}, \ref{lem-H-ineq-2}, and \ref{lem-H-ineq-3}, we get $L(A_j)\leq L(E_j^0)$. Since the mapping $T_2T_1$ in Definition \ref{def-generalized-Awesome} is an isometry, we get $L(A_k^j)\leq L(E_l^j)$, as required.
\end{proof}

\begin{cor}\label{cor-network}
Geodesics in $X_H$ are network paths.
\end{cor}
\begin{proof}
Suppose $P$ is a shortest path between $I_j$ and $I_k$ but is not a network path.  Then there exists a part of $P$, say $P_i$, such that $P_i$ is not a network part.  By theorem \ref{thm-faster}, we know the Awesome path $A_k^j$ between these intervals is shorter in length than $P_i$, the shortest Euclidean path between intervals.  Now $P' = \prod_{l=1}^n P_l'$ where $$P_l' =
\begin{cases}
      A_k^j &\text{if }l=i \\
      P_l &\text{otherwise}
\end{cases}$$ is a shorter path.
\end{proof}

\subsection{Monotone Network Paths}

\begin{defn}
A {\bf monotone network path} $P$ is a path in which all directed segments are either
\begin{itemize}
    \item $\uparrow, \rightarrow, \nearrow$
    \item $\uparrow, \leftarrow, \nwarrow$
    \item $\downarrow, \leftarrow, \swarrow$
    \item or $\downarrow, \rightarrow, \searrow$.
\end{itemize}
\end{defn}

It can be shown that when $j_1\neq 0$ that any geodesic from $I_0$ to $I_j$ is a monotone network path.

\begin{lem}\label{lem-combLength=indexDepth}
If $P$ is a monotone network path from $I_0$ to $I_j$ with combinatorial length $n$, then $n = d(j)$, where {\bf depth} of $j$, denoted $d(j)$, is defined as $d(j)= \frac{|j_1|+|j_2|}{3}$.
\end{lem}
\begin{proof}
With out loss of generality, suppose $0 < j_1$ and $0 \leq j_2$.  When $n = 1$, $P = \nearrow_{(1.5,1.5)} \text{ or } \rightarrow_{(3,0)}$.  In both scenarios, $d(j) =1$.

Fix $n$.
Suppose for any $j\in J$ and any MNP from $I_0$ to $I_j$ with combinatorial length $n$ we have $n = d(j)$.  Let $P$ be a MNP from $I_0$ to $I_j$ of combinatorial length $n+1$.  Write $\prod_{i=1}^{n+1}P_i$ as the part decomposition of $P$. Notice $P' = \prod_{i=1}^{n}P_i$ is an MNP from $I_0$ to $I_{j'}$ for some $j' \in J$ and the combinatorial length of $P'$ is $n$.  Also the final part $P_{n+1}$ is either $\nearrow, \rightarrow, \text{ or } \uparrow$, since $P$ is MNP.  Thus $$j' = j+(-1.5,-1.5), j+(-3,0), \text{ or } j + (0,-3).$$
In all scenarios $$d(j) = 
\begin{Bmatrix} 
    d(j' + (1.5,1.5)) &=& d(j') + d(1.5,1.5)\\
    d(j' + (3,0)) &=& d(j') + d(3,0)\\
    d(j' + (0,3)) &=& d(j') + d(0,3)
\end{Bmatrix} 
=d(j') + 1 = n + 1$$
\end{proof}

\begin{cor}
If $P$ is a monotone network path from $I_0$ to $I_j$ with exactly $D$ diagonals then $P$ has length 
\begin{equation*}
    L(P) = D \sqrt 2 + (d(j) - D) \cdot 2. \tag{C}
\end{equation*}
\end{cor}
\begin{proof}
Apply Proposition \ref{prop-networkPathLength} and Lemma \ref{lem-combLength=indexDepth}.
\end{proof}

Therefore,  a monotone network path $P$ from $I_0$ to $I_j$, with $0<j_1$ and $0<j_2$,  has the shortest length if it has the largest possible number of diagonals $D$. The maximum number of diagonals can be acquired by considering the largest square inscribed inside the restricting rectangle. This square can fit exactly $$D= \frac{2}{3}\min(|j_1|,|j_2|)$$ diagonals. Notice $$\min(a,b)=\frac{a+b-|a-b|}{2}$$  and recall $$d(j) =\frac{|j_1|+|j_2|}{3} $$ Plugging these values into Formula (C) yields:

\begin{thm}\label{thm:d_Hformula}
The distance between intervals $I_0$ and $I_j$ in $X_H$  is $$d_H(I_0,I_j) = \frac{\sqrt 2}{3}(|j_1|+|j_2|)+ \frac{2 -\sqrt 2}{3}||j_1|-|j_2|| $$ provided $j_1\neq 0$.
When $j_1=0,$ we have 
$$d_H(I_0,I_j)=2\sqrt{2} + \frac{2}{3}|j_2| -2.$$
\end{thm}

\section{\label{sec:tancone}The Tangent Cone at Infinity of $X_H$}
As we mentioned in the introduction, we will use the main theorem in \cite{Sormani-Kazaras-Students1} to prove that tangent cone at infinity of $X_H$ is unique and is a norm vector space with the norm $F:\mathbb{E}^2\rightarrow[0,\infty)$ given in equation \ref{equ:norm}:

$$F(x_1, x_2) = \frac{\sqrt 2 }{3}(|x_1|+|x_2|) + \frac{2-\sqrt 2}{3}||x_1|-|x_2||.$$

The theorem is as follows:
\begin{thm}\label{thm-smocking-R}
Suppose we have an $N$ dimensional smocked space, $(X, d)$, as in Definition~\ref{defn-smock}
 such that
\be
\left|\,\bar{d}(x, x')\,- \,[F(x)-F(x')] \,\right| \,\le \, K \qquad \forall x,x' \in {\mathbb{E}}^N
\ee
where $F: {\mathbb{E}}^N \to [0,\infty)$ is a norm.
Then $(X,d)$ has a unique tangent cone at infinity, 
\be
({\mathbb{R}}^N, d_F) \textrm{ where }
d_F(x,x')=||x-x'||_F=F(x-x').
\ee
\end{thm}

To check for the conditions of this theorem, we use the following approximation lemma from paper \cite{Sormani-Kazaras-Students1}.

\begin{lem} \label{lem-dil-to-approx}
Given an $N$ dimensional smocked space 
parametrized by points in intervals as in (\ref{param-by-points}),
with smocking depth, $h\in (0,\infty)$,
and smocking length $L=L_{max}\in (0,\infty)$, 
if one can find a Lipschitz
function $F: {\mathbb{E}}^N \to [0, \infty)$ such that
\be
|\,d(I_j, I_{j'})\,- \,[F(j)-F(j')] \,| \,\le \,C, 
\ee
then
\be
|\,\bar{d}(x, x')\,- \,[F(x)-F(x')] \,| \,\le \, 2h+ C + 2 \dil(F) (h+L)
\ee
where $dil(F)$ is the dilation factor or Lipschitz constant of $F$:
\be
\dil(F) = \sup \left\{\frac{|F(a)-F(b)|}{|a-b|}\,:\,\, a\neq b \in {\mathbb E}^N\right\}.
\ee
\end{lem}
In what follows, we show that $F$ as defined by equation \ref{equ:norm} is a norm that satisfies the conditions of theorem \ref{thm-smocking-R}.

\begin{lem}
$F$ is a norm.
\end{lem}

\begin{proof}
\begin{enumerate}
    \item Note, for any $v\in \mathbb{E}^2$, $F(v)\geq 0$. Also, if $F(v) = 0$ then $|v_1| + |v_2| = 0$, and thus $v=0$.
    \item Notice, for any $a\in \mathbb{R}$ and any $v\in \mathbb{E}^2$,  we have
    \begin{align*}
        F(av) &= \frac{\sqrt 2 }{3}(|av_1|+|av_2|) + \frac{2-\sqrt 2}{3}||av_1|-|av_2|| \\&= |a|(\frac{\sqrt 2 }{3}(|v_1|+|v_2|) + \frac{2-\sqrt 2}{3}||v_1|-|v_2||) = |a|F(v).
    \end{align*}
     \item Let $S=F^{-1}([0,1])$. Note that $S$ is a convex regular octagon. Observe that for any $v\in\mathbb{E}^2$, we have $F(v)=\inf A_v$ where $A_v$ is defined to be $A_v = \{t\in \mathbb{R}_+ \vert \tfrac{v}{t}\in S\}$.  Notice that for any $u,v\in \mathbb{E}^2$, we have that $A_u + A_v \subset A_{u+v}$:
     \begin{itemize}
         \item[] 
         Let $t\in A_u$ and $t' \in A_v$. Then $\tfrac{u}{t},\tfrac{v}{t'}\in S$.  Since $S$ is convex and $0\leq \tfrac{t}{t+t'}\leq 1$, we have $\frac{u+v}{t+t'} = \left(\frac{t}{t+t'}\right)\left(\frac{u}{t}\right) + \left(1 - \frac{t}{t+t'}\right)\left(\frac{v}{t'}\right)\in S$. Hence $t+t'\in A_{u+v}$. 
     \end{itemize}
     Thus, $F(u+v) = \inf A_{u+v} \leq \inf (A_u + A_v) = \inf A_u + \inf A_v = F(u)+F(v)$.
\end{enumerate}
\end{proof}

\begin{lem}\label{assumption1}
We have $|d_H(I_j, I_{j'}) - [F(j)-F(j')]|\leq 2\sqrt{2}-2$.
\end{lem}
\begin{proof}
Recall  that in corollary \ref{thm:d_Hformula}, we proved $$d_H(I_0,I_j) = 
\begin{cases}
      \frac{\sqrt 2}{3}(|j_1|+|j_2|)+ \frac{2 -\sqrt 2}{3}||j_1|-|j_2|| & \text{if } j_1 \neq 0,\\
      2\sqrt{2} + \frac{2}{3}|j_2| -2 & \text{if } j_1 = 0.\\
\end{cases}$$

Hence $$F(j) \leq d_H(I_0,I_j) \leq F(j) +2\sqrt{2}-2.$$

Therefore $$F(j) \leq d_H(I_0,I_j) \leq d_H(I_0,I_{j'}) + d_H(I_j,I_{j'}) \leq  d_H(I_j,I_{j'}) + F(j') +2\sqrt{2}-2$$ and hence $|d_H(I_j, I_{j'}) - [F(j)-F(j')]|\leq 2\sqrt{2}-2$ as required.
\end{proof}

\begin{lem}\label{assumption2}
$F: \mathbb{E}^2 \rightarrow [0,\infty)$ is Lipschitz.
\end{lem}

\begin{proof}
We need to show $\forall x,y \in \mathbb{E}^2$, $$\frac{|F(x)-F(y)|}{|x-y|} \le k,$$ for a constant $k$ independent of $x$ and $y$.

Without loss of generality, we may assume $F(y)\leq F(x)$.

Let $\alpha = \frac{\sqrt{2}}{3}$ and $\beta= \frac{2-\sqrt2}{3}$. Then,
$$F(x) = \alpha (|x_1|+|x_2|) + \beta ||x_1|-|x_2||,$$
and
$$F(y) = \alpha (|y_1|+|y_2|) + \beta ||y_1|-|y_2||.$$
Therefore,
\begin{align*}
F(x) - F(y) &= \alpha (|x_1|-|y_1|) + \alpha (|x_2|-|y_2|)+ \beta (||x_1|-|x_2||- ||y_1|-|y_2||) \\
&\leq \alpha (|x_1|-|y_1|) + \alpha (|x_2|-|y_2|)+ \beta (||x_1|-|x_2|- |y_1|+|y_2||) \\
&= \alpha (|x_1|-|y_1|) + \alpha (|x_2|-|y_2|)+ \beta (||x_1|-|y_1| +|y_2|-|x_2||) \\
&\leq \alpha (|x_1|-|y_1|) + \alpha (|x_2|-|y_2|)+ \beta ||x_1|-|y_1|| +\beta||x_2|-|y_2|| \\
&\leq \alpha |x_1-y_1| + \alpha|x_2-y_2|+ \beta |x_1-y_1| +\beta|x_2-y_2| \\
&= (\alpha + \beta )|x_1-y_1| + (\alpha + \beta )|x_2-y_2| \\
&= (\alpha + \beta )(|x_1-y_1| +|x_2-y_2|) \\
&\leq  (\alpha + \beta)|x-y|\sqrt 2.\
\end{align*} 
Hence, $k=\frac{2\sqrt{2}}{3}.$

\end{proof}

\begin{lem} \label{dil-lemma} We have 
$dil(F) \leq \frac{2+\sqrt2}{3}$.
\end{lem}
\begin{proof}
Since $dil(f+g)\leq dil (f) + dil(g)$ and $dil(|x|)= 1$, and $dil(f\circ g)= dil(f) dil(g)$, we have:

\begin{eqnarray*}
dil(F)&=& dil(\frac{\sqrt{2}}{3} (|a_1|+|a_2|) + \frac{2-\sqrt{2}}{3}||a_1|-|a_2||)\\
&\leq& dil(\frac{\sqrt{2}}{3} (|a_1|+|a_2|)) +dil(\frac{2-\sqrt{2}}{3}||a_1|-|a_2||)\\
&=&  \frac{\sqrt{2}}{3} dil( |a_1|+|a_2|) +\frac{2-\sqrt{2}}{3} dil(||a_1|-|a_2||)\\
&\leq& \frac{\sqrt{2}}{3} dil( |a_1|)+dil(|a_2|) +\frac{2-\sqrt{2}}{3} dil(|\cdot|)dil(|a_1|-|a_2|)\\
&\leq& \frac{2\sqrt{2}}{3}  +\frac{2-\sqrt{2}}{3}= \frac{2+\sqrt2}{3}.
\end{eqnarray*}
\end{proof}

Lemmas \ref{assumption1} and \ref{assumption2} satisfy the assumption of Lemma \ref{lem-dil-to-approx}. By Lemma \ref{lem-dil-to-approx} we have $$|\bar{d}_H(x,x') - [F(x)-F(x')]| \leq 2h_H + C + 2 dil(F) (h_H+L_H),$$  where $h_H= 1.5$, $C = 2\sqrt{2}-2$ and $L_H = 1$.  Combining with \ref{dil-lemma}, we get $$|\bar{d}_H(x,x') - [F(x)-F(x')]| \leq 2(1.5) + (2\sqrt{2}-2) + 2 (\frac{2+\sqrt2}{3}) (1.5+1) = \frac{18+11\sqrt{2}}{3}.$$

Finally, by Theorem \ref{thm-smocking-R}, we get that the H-smocking space $(X_H, d_H)$ has a unique tangent cone at infinity $(\mathbb{E}^2, d_F)$ where $d_F(x,x')=F(x-x')$.

{\bf Acknowledgements:}   We are grateful to Professor Christina Sormani and Dr. Demetre Kazaras for supervising our project and their many helpful comments. We gratefully acknowledge the {\em Graduate Center of the City University of 
New York} and the {\em Simons Center for Geometry and Physics at Stony Brook University} for providing space where we could meet to conduct this research.
\bibliographystyle{plain}
\bibliography{checkeredpaper.bbl}

\end{document}